\pdfoutput=1
\RequirePackage{ifpdf}
\ifpdf
\documentclass[pdftex]{sigma}
\else
\documentclass{sigma}
\fi

\usepackage{tikz}
\usetikzlibrary{decorations.markings}
\usetikzlibrary{arrows}
\usepackage{multicol}

\numberwithin{equation}{section}
\newtheorem{Theorem}{Theorem}[section]
\newtheorem{Corollary}[Theorem]{Corollary}
\newtheorem{Lemma}[Theorem]{Lemma}
\newtheorem{Proposition}[Theorem]{Proposition}
{
 \theoremstyle{definition}
 \newtheorem{Definition}[Theorem]{Definition}
 \newtheorem{Remark}[Theorem]{Remark}
}

\newcommand{\cA}{\mathcal{A}}

\newcommand{\ZZ}{\mathbb{Z}}
\newcommand{\bg}{\mathbf{g}}
\newcommand{\bc}{\mathbf{c}}
\newcommand{\bb}{\mathbf{b}}
\newcommand{\id}{\mathrm{id}}

\newcommand{\dynkinradius}{.06cm}
\newcommand{\dynkinstep}{.7cm}
\newcommand{\dynkinlinesep}{.08cm}
\newcommand{\dynkindot}[4]{\fill (\dynkinstep*#1,\dynkinstep*#2) circle (\dynkinradius); \node at (\dynkinstep*#1,\dynkinstep*#2) [label={[label distance=\dynkinradius, label position=#4]\tiny $#3$}] {};}
\newcommand{\dynkinline}[4]{\draw[thin] (\dynkinstep*#1,\dynkinstep*#2) -- (\dynkinstep*#3,\dynkinstep*#4);}
\newcommand{\dynkindotline}[4]{\draw[dotted] (\dynkinstep*#1,\dynkinstep*#2) -- (\dynkinstep*#3,\dynkinstep*#4);}
\newcommand{\dynkindoubleline}[4]{\draw[thin,double distance=\dynkinlinesep,postaction={decorate}] (\dynkinstep*#1,\dynkinstep*#2) -- (\dynkinstep*#3,\dynkinstep*#4);}
\newcommand{\dynkintripleline}[4]{\draw[thin, double distance=\dynkinlinesep*1.5,postaction={decorate}] (\dynkinstep*#1,\dynkinstep*#2) -- (\dynkinstep*#3,\dynkinstep*#4); \draw[thin] (\dynkinstep*#1,\dynkinstep*#2) -- (\dynkinstep*#3,\dynkinstep*#4);}
\newenvironment{dynkin}{\begin{tikzpicture}[baseline=-0.6ex,decoration={markings,mark=at position 0.7 with {\arrow[line width=0.06cm]{angle 60}}}]}{\end{tikzpicture}}

\begin{document}

\allowdisplaybreaks

\newcommand{\arXivNumber}{1604.06286}

\renewcommand{\PaperNumber}{067}

\FirstPageHeading

\ShortArticleName{Exchange Relations for Finite Type Cluster Algebras}

\ArticleName{Exchange Relations for Finite Type Cluster Algebras\\ with Acyclic Initial Seed and Principal Coef\/f\/icients}

\Author{Salvatore STELLA~$^\dag$ and Pavel TUMARKIN~$^\ddag$}

\AuthorNameForHeading{S.~Stella and P.~Tumarkin}

\Address{$^\dag$~IN$d$AM - Marie Curie Actions fellow, Universit\`a ``La Sapienza'', Roma, Italy}
\EmailD{\href{mailto:tella@mat.uniroma1.it}{stella@mat.uniroma1.it}}
\URLaddressD{\url{http://www1.mat.uniroma1.it/people/stella/index.shtml}}

\Address{$^\ddag$~Department of Mathematical Sciences, Durham University, UK}
\EmailD{\href{mailto:pavel.tumarkin@durham.ac.uk}{pavel.tumarkin@durham.ac.uk}}
\URLaddressD{\url{http://www.maths.dur.ac.uk/users/pavel.tumarkin/}}

\ArticleDates{Received April 22, 2016, in f\/inal form July 06, 2016; Published online July 09, 2016}

\Abstract{We give an explicit description of all the exchange relations in any f\/inite type cluster algebra with acyclic initial seed and principal coef\/f\/icients.}


\Classification{13F60}

\section{Introduction and main results}
A cluster algebra, as def\/ined by Fomin and Zelevinsky in~\cite{FZ02}, is a commutative ring with a distinguished set of generators called \emph{cluster variables}. Cluster variables are grouped into overlapping collections of the same cardinality (\emph{clusters}) connected by local transition rules called \emph{mutations}. To each mutation corresponds an \emph{exchange relation}: a dependency relation among the cluster variables of two adjacent clusters.
In~\cite{FZ03}, Fomin and Zelevinsky showed that cluster algebras of \emph{finite type}, i.e., those containing only a f\/inite number of cluster variables, are classif\/ied by f\/inite type Cartan matrices.

Given two cluster variables in a cluster algebra, deciding whether they belong to the same cluster or if they can be obtained from one another by a single mutation is, in general, a~hard problem to address. In several special situations though, when suitable combinatorial models exist, such questions become much easier to decide. This is the case, for example, of cluster algebras originating from marked surfaces~\cite{FST08,FT12} and orbifolds~\cite{FST12}, or those having an asso\-ciated cluster category. Here we will consider the case of cluster algebras of f\/inite type with an acyclic initial seed where the answer can be given uniformly using the \emph{compatibility degree} of the corresponding \emph{$\bg$-vectors}.

Knowing that two cluster variables are exchangeable naturally arises the problem of producing the exchange relation they satisfy. Answers to this question exist depending on the available models; for instance, in the surfaces case, these can be expressed in terms of \emph{skein relations}~\cite{MW13}, while for cluster categories one can leverage the \emph{multiplication formula} of \cite{CK08}.

In \cite{YZ08}, using some determinantal identities on the associated Lie group, the authors were able to give explicit formulas for all the \emph{primitive} exchange relations (i.e., those in which cluster variables only appear in one of the two monomials of the right hand side) in any cluster algebra of f\/inite type with an acyclic initial seed. Their recipe works for principal coef\/f\/icients and hence, via separation of additions, for any other choice of coef\/f\/icients. In~\cite{Ste13} the f\/irst author gave a~uniform formula for all the exchange relations in the same class of algebras, albeit only in the coef\/f\/icient-free case. The main goal of the current paper is to improve on this result to deal with principal coef\/f\/icients as well. Namely, given any two exchangeable cluster variables in a~f\/inite type cluster algebra with acyclic initial seed and principal coef\/f\/icients, we give an explicit formula for computing their exchange relation. This exchange relation has also a geometric interpretation in terms of roots and weights of the corresponding root system.

In order to make this more precise, we need to recall a few notions and results from \cite{Ste13,YZ08}.

Let $A=(a_{ij})$ be any f\/inite type Cartan matrix; we denote by $\Gamma$ its Dynkin diagram and by $W=\langle s_1,\dots,s_n\rangle$ the associated Weyl group and simple ref\/lections. To each \emph{Coxeter element} $c=s_{i_1}\cdots s_{i_n}$ in $W$ we can associate a skew-symmetrizable integer matrix $B_c=(b_{ij})_{i,j\in[1,n]}$ by setting
 \begin{gather*}
 b_{ij}= \begin{cases}
 -a_{ij} & \text{if } i\prec_c j, \\
 a_{ij} & \text{if } j\prec_c i, \\
 0 & \text{otherwise},
 \end{cases}
 \end{gather*}
where we write $i\prec_c j$ if and only if $s_i$ precedes $s_j$ in all reduced expressions of $c$. As $c$ varies, we get all possible \emph{acyclic} exchange matrices of the same mutation class. We will denote by $\cA_\bullet(c)$ the cluster algebra with initial exchange matrix $B_c$ and \emph{principal coefficients} at the initial seed.

The algebra $\cA_\bullet(c)$ is $\ZZ^n$-graded; its cluster variables and cluster monomials are homogeneous elements and their \emph{$\bg$-vector} is their homogeneous degree (see~\cite[Section~6]{FZ07}). Let $\omega_i$ be the $i$-th \emph{fundamental weight} in the \emph{weight lattice}~$P$ of~$\Gamma$; we will routinely interpret $\bg$-vectors as weights by writing them in the basis of fundamental weights.

 Let $w_0$ be the longest element of $W$ and denote by $h(i;c)$ the minimum positive integer such that
 \begin{gather*}
 c^{h(i;c)}\omega_i = w_0\omega_i
 \end{gather*}
 (it is a f\/inite number \cite[Proposition~1.3]{YZ08}).
 \begin{Theorem}[{\cite[Theorem~1.4]{YZ08}}]
 The cluster variables of $\cA_\bullet(c)$ are naturally in bijection with the elements of the set
 \begin{gather*}
 \Pi(c) := \big\{ c^m\omega_i \colon i\in[1,n] , \, 0\leq m \leq h(i;c) \big\}.
 \end{gather*}
 To the cluster variable $x_\lambda$ it corresponds its $\bg$-vector $\lambda\in\Pi(c)$.
 \end{Theorem}

This correspondence extends to a bijection between points of $P$ and \emph{cluster monomials} of~$\cA_\bullet(c)$ (cf.~\cite[Theorem~1.2]{Ste13}); for~$\lambda\in P$ we will denote by~$x_\lambda$ the cluster monomial whose $\bg$-vector is~$\lambda$.

The set $\Pi(c)$ is naturally endowed with a permutation $\tau_c$ def\/ined by
 \begin{gather*}
 \tau_c (\lambda) := \begin{cases}
 \omega_i & \text{if $\lambda = -\omega_i$}, \\
 c\lambda & \text{otherwise},
 \end{cases}
 \end{gather*}
 which extends to a piecewise linear map on the whole of $P$ that is ``compatible'' with the cluster structure of $\cA_\bullet(c)$.
 This is a combinatorial shadow of a notable automorphism of the coef\/f\/icient-free counterpart of $\cA_\bullet(c)$ sending the cluster variable $x_\lambda$ to $x_{\tau_c(\lambda)}$.

 Let $Q$ be the \emph{root lattice} of $\Gamma$ with simple roots $\alpha_i$; as for $P$, we will routinely think of elements in~$Q$ as integer vectors using the basis of simple roots.
 \begin{Definition}
 The \emph{compatibility degree} $(\cdot||\cdot)_c$ is the unique $\tau_c$-invariant function on pairs of elements of $\Pi(c)$ def\/ined by the initial conditions
 \begin{gather*}
 (\omega_i||\lambda)_c := \big[ \big(c^{-1}-\id\big)\lambda ; \alpha_i\big]_+,
 \end{gather*}
 where, for $v$ in $Q$, $[v;\alpha_i]$ denotes the $i$-th coef\/f\/icient of $v$ and $[m]_+$ is a shorthand for $\max\{m, 0\}$ (cf.~\cite[Proposition~5.1]{YZ08}).
 \end{Definition}

 The name comes from the following important property, consequence of the polytopal reali\-za\-tion of the cluster fan of~$\cA_\bullet(c)$ \cite{CFZ02,Ste13}.
 \begin{Proposition}
 Two weights $\lambda$ and $\mu$ from $\Pi(c)$ are
 \begin{itemize}\itemsep=0pt
 \item
compatible $($i.e., there is a cluster of $\cA_\bullet(c)$ containing both~$x_\lambda$ and~$x_\mu)$ if and only if
 \begin{gather*}
 (\lambda||\mu)_c = 0 \qquad
 \text{$($equivalently $(\mu||\lambda)_c=0)$,}
 \end{gather*}

 \item
exchangeable $($i.e., there are two clusters of $\cA_\bullet(c)$ that can be obtained from one-another by swapping~$x_\lambda$ for~$x_\mu)$ if and only if
 \begin{gather*}
 (\lambda||\mu)_c = 1 = (\mu||\lambda)_c.
 \end{gather*}
 \end{itemize}
 \end{Proposition}

 Our starting point is the following restatement of \cite[Proposition 5.1]{Ste13}.
\begin{Proposition}
Suppose $\lambda$ and $\mu$ are exchangeable weights in $\Pi(c)$. Then the set
\begin{gather*}
 \big\{ \tau_c^{-m}\big(\tau_c^m(\lambda)+\tau_c^m(\mu)\big) \big\}_{m\in\ZZ}
 \end{gather*}
consists precisely of two weights. One of them is $\lambda+\mu$; denote the other by $\lambda\uplus_c\mu$.
 \end{Proposition}

Let $y_1,\dots,y_n$ be the generators of the coef\/f\/icient semif\/ield of $\cA_\bullet(c)$ and denote by $y^\alpha$ the product $\prod\limits_{i=1}^n y_i^{[\alpha;\alpha_i]}$.

 \begin{Theorem} \label{thm:main} Suppose $\lambda$ and $\mu$ are exchangeable weights in $\Pi(c)$. Then there exists a unique positive root $\alpha$ in the root system of $\Gamma$ such that
 \begin{gather} \label{eq:system}
 -B_c\alpha = \lambda+\mu-\lambda\uplus_c\mu
 \end{gather}
 and
 \begin{gather} \label{eq:dual}
 \langle\lambda,\alpha^\vee\rangle \langle\mu,\alpha^\vee\rangle = -1
 \end{gather}
 $($here $\alpha^\vee$ denotes the \emph{coroot} corresponding to $\alpha$ while $\langle\cdot\,,\cdot\rangle$ is the pairing of dual vector spaces$)$. Moreover the cluster variables $x_\lambda$ and $x_\mu$ of $\cA_\bullet(c)$ satisfy the exchange relation
 \begin{gather} \label{eq:exrel}
 x_\lambda x_\mu = x_{\lambda+\mu} + y^\alpha x_{\lambda\uplus_c\mu}.
 \end{gather}
 \end{Theorem}

Note that the shape of equation~\eqref{eq:exrel} follows immediately from the coef\/f\/icient-free case \cite[Proposition~5.2]{Ste13} together with the observations that $\bc$-vectors are roots in the root system of~$\Gamma$ (cf.~\cite{NS14}), and that the exchange relations in~$\cA_\bullet(c)$ are homogeneous. The real content of our theorem are therefore the explicit conditions~\eqref{eq:system} and~\eqref{eq:dual} that determine $\alpha$. They are clearly both necessary.

Indeed, equation~\eqref{eq:system} is just a restatement of the fact that the exchange relations in $\cA_\bullet(c)$ are homogeneous and that the degree of $y_i$ is $-\bb_i$ (the negative of the i-th column of~$B_c$). Equation~\eqref{eq:dual}, instead, follows immediately from \cite[equation~(1.11)]{NZ12} once we interpret $\bg$-vectors as weights and $\bc$-vectors as roots together with the observation that, when mutating in direction~$k$, the $k$-th $\bc$-vector changes into its negative.

On the other hand, equation~\eqref{eq:system} is not, in principle, suf\/f\/icient on its own because~$B_c$ is, in general, not invertible. Nonetheless, thanks to the fact that we are dealing with positive roots, we will see that it is still enough in every case except in type $D_n$.

\begin{Remark} \label{geom}
Equation~\eqref{eq:exrel} has the following geometric interpretation. Associating the $\bg$-vectors with weights, one can observe that every cluster corresponds to a cone with facets being mirrors of ref\/lections of the associated Weyl group (see Fig.~\ref{fig:exchange_relation}). If two clusters are neighbors in the exchange graph (i.e., they dif\/fer only by two exchangeable cluster variables $x_\lambda$ and $x_\mu$), then the corresponding cones share a facet, and this facet is precisely the mirror of the ref\/lection in the root~$\alpha$ from equation~\eqref{eq:exrel}.
 \begin{figure}[t] \centering
 \includegraphics[scale=0.4]{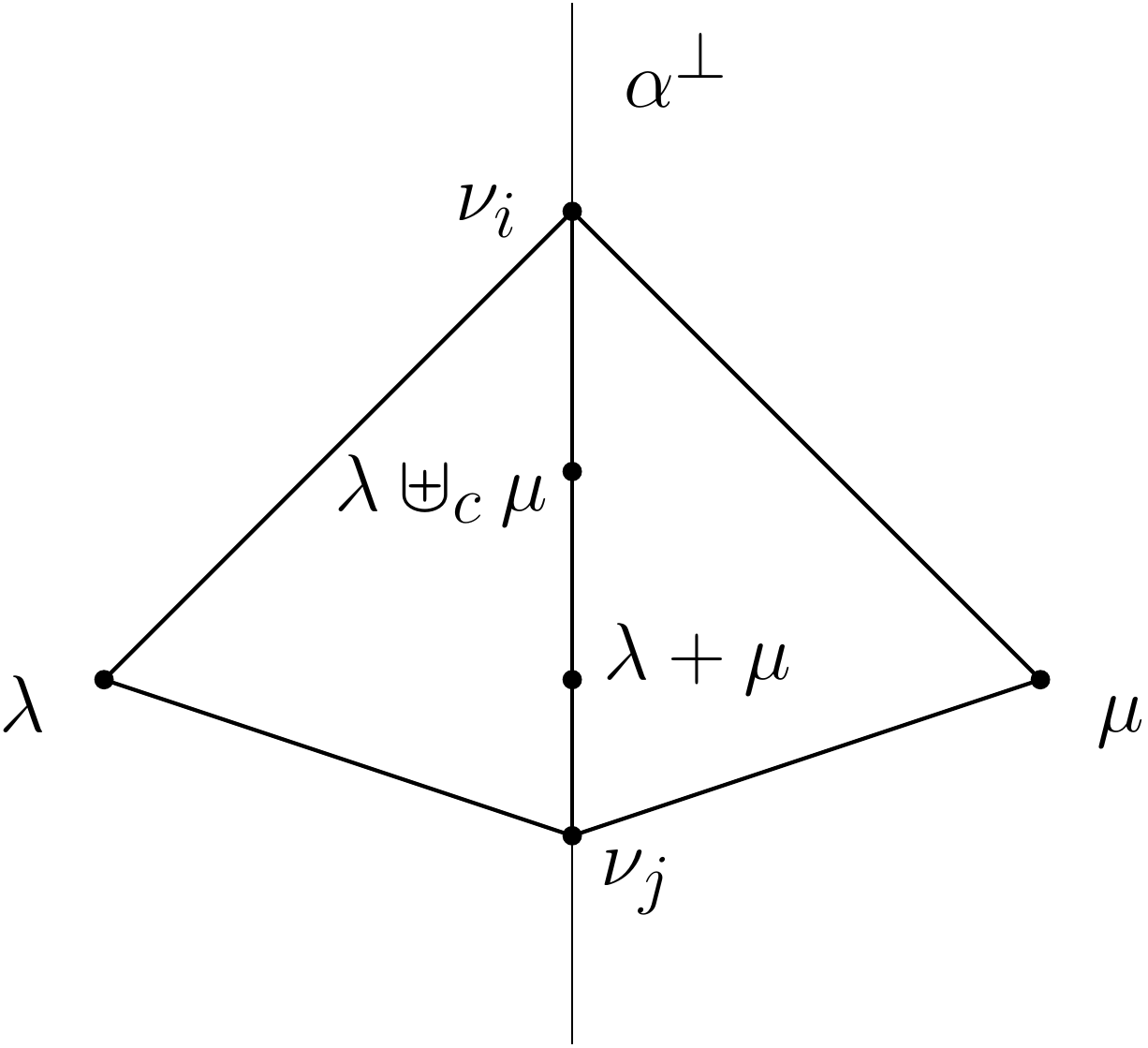}
 \caption{Geometric interpretation of exchange relations: $x_\lambda$ and $x_\mu$ are exchangeable cluster variables, $\lambda$ and $\mu$ their $\bg$-vectors, and $\alpha^\perp$ is the wall in the cluster fan separating the two clusters.} \label{fig:exchange_relation}
 \end{figure}
\end{Remark}

\section{Proof of Theorem~\ref{thm:main}}\label{proof}
Recall the notation for $\Gamma$, $c$ and $B_c$ from the previous section. Without loss of generality we consider only Dynkin diagrams that are connected. Further, we assume that the nodes of~$\Gamma$ are labeled according to the conventions in Fig.~\ref{fig:Dynkin}.

\begin{figure}[h!]
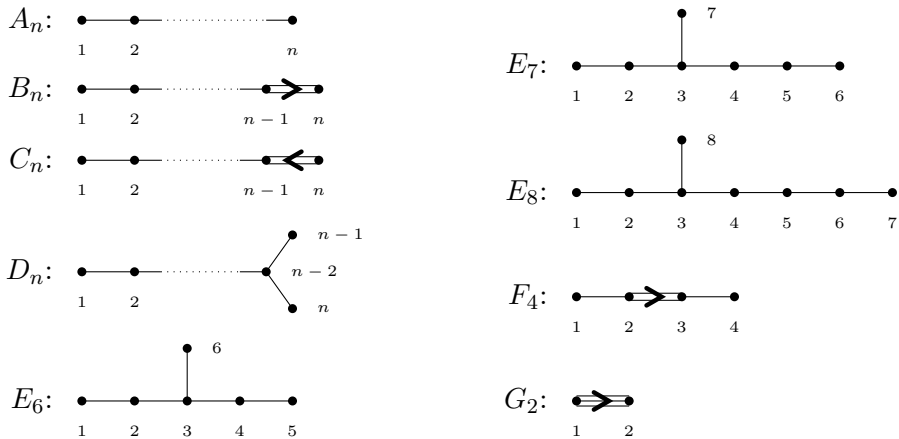
 \centering
 \begin{minipage}{0.8\textwidth}
 \begin{multicols}{2}
 \begin{itemize}\itemsep=0pt
 \item[$A_n$:]
 \begin{dynkin}
 \dynkinline{1}{0}{2.5}{0}
 \dynkinline{4.5}{0}{5}{0}
 \dynkindotline{2.5}{0}{4.5}{0}
 \dynkindot{1}{0}{1}{below}
 \dynkindot{2}{0}{2}{below}
 \dynkindot{5}{0}{\vphantom{1}n}{below}
 \end{dynkin}

 \item[$B_n$:]
 \begin{dynkin}
 \dynkinline{1}{0}{2.5}{0}
 \dynkinline{4}{0}{4.5}{0}
 \dynkindotline{2.5}{0}{4}{0}
 \dynkindoubleline{4.5}{0}{5.5}{0}
 \dynkindot{1}{0}{1}{below}
 \dynkindot{2}{0}{2}{below}
 \dynkindot{4.5}{0}{n-1}{below}
 \dynkindot{5.5}{0}{\vphantom{1}n}{below}
 \end{dynkin}

 \item[$C_n$:]
 \begin{dynkin}
 \dynkinline{1}{0}{2.5}{0}
 \dynkinline{4}{0}{4.5}{0}
 \dynkindotline{2.5}{0}{4}{0}
 \dynkindoubleline{5.5}{0}{4.5}{0}
 \dynkindot{1}{0}{1}{below}
 \dynkindot{2}{0}{2}{below}
 \dynkindot{4.5}{0}{n-1}{below}
 \dynkindot{5.5}{0}{\vphantom{1}n}{below}
 \end{dynkin}

 \item[$D_n$:]
 \begin{dynkin}
 \dynkinline{1}{0}{2.5}{0}
 \dynkinline{4}{0}{4.5}{0}
 \dynkindotline{2.5}{0}{4}{0}
 \dynkinline{5}{0.7}{4.5}{0}
 \dynkinline{5}{-0.7}{4.5}{0}
 \dynkindot{1}{0}{1}{below}
 \dynkindot{2}{0}{2}{below}
 \dynkindot{4.5}{0}{n-2}{right}
 \dynkindot{5}{0.7}{n-1}{right}
 \dynkindot{5}{-0.7}{n}{right}
 \end{dynkin}

 \item[$E_6$:]
 \begin{dynkin}
 \dynkinline{1}{0}{5}{0}
 \dynkinline{3}{0}{3}{1}
 \dynkindot{1}{0}{1}{below}
 \dynkindot{2}{0}{2}{below}
 \dynkindot{3}{0}{3}{below}
 \dynkindot{4}{0}{4}{below}
 \dynkindot{5}{0}{5}{below}
 \dynkindot{3}{1}{6}{right}
 \end{dynkin}

 \item[$E_7$:]
 \begin{dynkin}
 \dynkinline{1}{0}{6}{0}
 \dynkinline{3}{0}{3}{1}
 \dynkindot{1}{0}{1}{below}
 \dynkindot{2}{0}{2}{below}
 \dynkindot{3}{0}{3}{below}
 \dynkindot{4}{0}{4}{below}
 \dynkindot{5}{0}{5}{below}
 \dynkindot{6}{0}{6}{below}
 \dynkindot{3}{1}{7}{right}
 \end{dynkin}

 \item[$E_8$:]
 \begin{dynkin}
 \dynkinline{1}{0}{7}{0}
 \dynkinline{3}{0}{3}{1}
 \dynkindot{1}{0}{1}{below}
 \dynkindot{2}{0}{2}{below}
 \dynkindot{3}{0}{3}{below}
 \dynkindot{4}{0}{4}{below}
 \dynkindot{5}{0}{5}{below}
 \dynkindot{6}{0}{6}{below}
 \dynkindot{7}{0}{7}{below}
 \dynkindot{3}{1}{8}{right}
 \end{dynkin}
 \vfill

 \item[$F_4$:]
 \begin{dynkin}
 \dynkinline{1}{0}{2}{0}
 \dynkindoubleline{2}{0}{3}{0}
 \dynkinline{3}{0}{4}{0}
 \dynkindot{1}{0}{1}{below}
 \dynkindot{2}{0}{2}{below}
 \dynkindot{3}{0}{3}{below}
 \dynkindot{4}{0}{4}{below}
 \end{dynkin}
 \vfill

 \item[$G_2$:]
 \begin{dynkin}
 \dynkintripleline{1}{0}{2}{0}
 \dynkindot{1}{0}{1}{below}
 \dynkindot{2}{0}{2}{below}
 \end{dynkin}

 \end{itemize}
 \end{multicols}
 \end{minipage}
 \caption{Finite type Dynkin diagrams.}
 \label{fig:Dynkin}
\end{figure}

 We begin our analysis with some easy considerations on the rank of $B_c$.
 \begin{Lemma} \label{lem:dimensions} If the type of $\Gamma$ is not $D_n$, then the kernel of $B_c$ has dimension $0$ if~$n$ is even and~$1$ if $n$ is odd.
 If the type of~$\Gamma$ is~$D_n$, then the kernel of~$B_c$ has dimension~$2$ if~$n$ is even and~$1$ if~$n$ is odd.
 \end{Lemma}

\begin{proof}The rank of $B_c$ is invariant under mutations so it suf\/f\/ices to establish the property for a single choice of~$c$. Let then $c=s_1\cdots s_n$ so that all the positive entries of $B_c$ are above the main diagonal. Exceptional types could be dealt uniformly in the argument at the expense of introducing heavier notation. We prefer to check the lemma by direct inspection in those cases.

Assume at f\/irst that the type of $\Gamma$ is not $D_n$. When $n$ is even, the matrix $B_c$ is invertible. Indeed, expanding by the f\/irst column and then by the f\/irst row, we get
 \begin{gather*}
 \det(B_c)=\det(B_c'),
 \end{gather*}
where $B_c'$ is a $(n-2)\times(n-2)$ matrix in the same inf\/inite class of $B_c$. The result follows then immediately by induction because all $2\times2$ skew-symmetrizable non-zero matrices are invertible. On the other hand, when $n$ is odd, $B_c$ being skew-symmetrizable implies immediately that $\det(B_c)=0$. Combining the two assertions we get that, for odd $n$, the dimension of the kernel of $B_c$ is~$1$.

To get the result in type $D_n$ it is enough to observe that the last two rows (and columns) of~$B_c$ are identical. We deduce therefore the required property from type $A_{n-1}$.
\end{proof}

This establishes Theorem~\ref{thm:main} whenever $n$ is even and the type of $\Gamma$ is not~$D_n$. In particular the result holds for all the exceptional types apart from type~$E_7$; in order to simplify the remaining analysis we check this case by hand. For each possible~$c$, the computations required amount to show that, whenever two positive roots satisfy equation~\eqref{eq:system}, only one of them satisf\/ies also equation~\eqref{eq:dual}; we omit the straightforward but lengthy calculations.

Alternatively, one could use the following observation to obtain type $E_7$ from type $E_8$.

\begin{Remark} \label{rk:induction}
A careful reader may observe that Lemma~\ref{lem:dimensions} could be used to establish Theo\-rem~\ref{thm:main} directly in all f\/inite types with the exception of~$D_n$. Indeed it would be enough to extend any $(2k+1) \times (2k+1)$ exchange matrix to a $(2k+2) \times (2k+2)$ exchange matrix of the same type and deduce the required property from the resulting algebra embedding. Instead, we prefer to give a more explicit argument that will simplify the analysis in type $D_n$ as well.
\end{Remark}

From now on we assume that $\Gamma$ is not of exceptional type. To deal with the remaining inf\/inite families we compute explicit generators for the kernel of~$B_c$. Our argument will hinge upon an explicit description of the possible dif\/ferences of positive roots; unfortunately in small rank non-generic situations may arise. We therefore verify Theorem~\ref{thm:main} by direct inspection in types~$A_3$, $B_3$, $C_3$, $D_4$ and~$D_6$; the calculations required are similar to those for type~$E_7$ and again we omit the details here.

\begin{Definition}
The \emph{support} of a vector $v$ is the full subdiagram of $\Gamma$ induced by the nodes corresponding to the non-zero coordinates of~$v$ when written in the basis of simple roots.
\end{Definition}

 \begin{Lemma}
Let $\Gamma$ be of type $A_n$, $B_n$, or $C_n$ with $n=2k+1$. Then the support of the vector spanning the kernel of~$B_c$ has exactly $k+1$ connected components.
 \end{Lemma}

 \begin{proof}
By Lemma~\ref{lem:dimensions} there is a unique (up to a scalar) non-zero vector $v$ such that $B_cv=0$. Since the only non-zero entries in~$B_c$ are located in the two diagonals adjacent to the main diagonal (cf.\ Fig.~\ref{fig:Dynkin}), $v$ is a linear combination of the~$\alpha_i$ with odd~$i$. Moreover, since $\Gamma$ is connected, all the entries of these two diagonals are non-zero so that all such~$\alpha_i$ appear with non-zero coef\/f\/icient and the claim follows.

More explicitly, for $i>2$ set
 \begin{gather*}
 \varepsilon_i := \begin{cases}
 1 & \text{if $i-2\prec_c i-1 \prec_c i$ or $i\prec_c i-1 \prec_c i-2$},\\
 -1 & \text{otherwise.}
 \end{cases}
 \end{gather*}
It is straightforward to verify that the kernel of $B_c$ is spanned by the vector $v$ def\/ined by
 \begin{gather} \label{eq:vector}
 v := \alpha_1 + \sum_{\substack{i\ \mathrm{odd}\\ 3\le i \leq n}} \frac{\varepsilon_i}{a_{i-1,i}} \alpha_i,
 \end{gather}
which proves the lemma.
 \end{proof}

 \begin{Lemma} \label{lem:ker_Dn_even}
 Let $\Gamma$ be of type $D_n$ with $n$ odd. Then the kernel of $B_c$ is generated by $\alpha_{n-1}+\alpha_n$ if $(n-1) \prec_c (n-2) \prec_c n$ or $n \prec_c (n-2) \prec_c (n-1)$. Otherwise it is generated by $\alpha_{n-1}-\alpha_n$.
 \end{Lemma}
\begin{proof}
Again by Lemma~\ref{lem:dimensions} there is a unique (up to a scalar) non-zero vector $v$ such that $B_cv=0$. The last two columns of $B_c$ are either identical (in which case $v=\alpha_{n-1}-\alpha_n$), or dif\/fer only in sign so that $v=\alpha_{n-1}+\alpha_n$.
\end{proof}

\begin{Lemma} Let $\Gamma$ be of type $D_n$ with $n=2k$ and $n\geq 4$. Then the kernel of~$B_c$ is generated by a vector whose support has exactly~$k$ connected components together with one of the two vectors~$\alpha_{n-1}\pm\alpha_n$ according to the same prescriptions of Lemma~{\rm \ref{lem:ker_Dn_even}}.
 \end{Lemma}
 \begin{proof}
The result follows directly by combining the previous two lemmas. Indeed the vector~\eqref{eq:vector} is killed by $B_c$ because it only interacts with a sub-matrix of type $A_{n-1}$ while the same reasoning of Lemma~\ref{lem:ker_Dn_even} applies to one of the two $\alpha_{n-1}\pm\alpha_n$. The two killed vectors are manifestly linearly independent.
 \end{proof}

 To use these information we need the following easy observation obtained by inspection of the appropriate list of roots.
 \begin{Lemma} \label{lem:components} If $\Gamma$ is of type $A_n$, $B_n$, or~$C_n$, then the support of the difference of any two positive roots in the root system of $\Gamma$ has at most two connected components. If $\Gamma$ is of type~$D_n$, then the support of the difference of any two positive roots in the root system of $\Gamma$ has at most three connected components.
 \end{Lemma}

 \begin{proof} We discuss type~$A_n$, the remaining types are obtained by similar considerations. In this case positive roots correspond to connected full subdiagrams of the associated Dynkin diagram. The support of the dif\/ference of two such roots $\alpha$ and $\beta$ is thus given by the symmetric dif\/ference of the support of $\alpha$ and the support of~$\beta$.
 \end{proof}

 \begin{Corollary} If $\Gamma$ is of type $A_n$, $B_n$, or $C_n$ with $n=2k+1\geq 5$, equation~\eqref{eq:system} has a unique solution among the positive roots of~$\Gamma$.
 \end{Corollary}
 \begin{proof}
 Indeed the dif\/ference of any two solutions is in the kernel of $B_c$ and this is generated by a vector with at least $k+1\geq3$ connected components.
 \end{proof}

 This concludes the proof of Theorem~\ref{thm:main} for types $A_n$, $B_n$ and $C_n$.

 \begin{Corollary} \label{cor:kernel-Dn} Suppose the type of $\Gamma$ is $D_n$ and $n\geq 7$. If equation~\eqref{eq:system} has more than one solution among the positive roots of $\Gamma$ then it has precisely two. Their difference is in the span of either one of $\alpha_{n-1}\pm\alpha_n$ depending on the relative order in which $s_{n-2}$, $s_{n-1}$, and~$s_n$ appear in~$c$.
 \end{Corollary}
 \begin{proof}
 The only possibility for two distinct roots to be solutions of equation~\eqref{eq:system} is for their dif\/ference to be in the span of $\alpha_{n-1}\pm\alpha_n$ (the other generating vector of the kernel of~$B_c$, when it exists, has too many connected components by the assumption on~$n$). We can conclude then by observing that positive roots in type~$D_n$ with such prescribed dif\/ference come in pairs.
 \end{proof}

To conclude the proof of Theorem~\ref{thm:main} it suf\/f\/ices to show that, in type $D_n$, whenever equation~\eqref{eq:system} is satisf\/ied by two roots only one of them verif\/ies equation~\eqref{eq:dual} as well. We will do so using the realization of cluster algebras via triangulations and laminations on surfaces introduced in~\cite{FST08,FT12} (see also~\cite{Sch08} for a detailed description of the model for~$D_n$ in the coef\/f\/icient-free case). The reader not familiar with the relevant terminology can f\/ind a simplif\/ied summary (suf\/f\/icient for the case at hand) in the beginning of \cite[Section~4.1]{NS14}.

Any cluster algebra of type $D_n$ can be realized as a triangulated once-punctured disk. Since we are only considering acyclic initial seeds, the collection of elementary laminations encoding the initial triangulation will contain a digon with one side on the boundary of the disk (cf.\ Fig.~\ref{fig:D_n-roots}). By ref\/lecting our surface, if necessary, we can always assume that $n-2 \prec_c n-1$; it will therefore suf\/f\/ice to consider only two cases: either $n-2 \prec_c n$, or $n \prec_c n-2$. Moreover we can always change simultaneously all the taggings and spiralling directions at the puncture to simplify our pictures.

 \begin{figure}[t]\centering
 \includegraphics[scale=0.6]{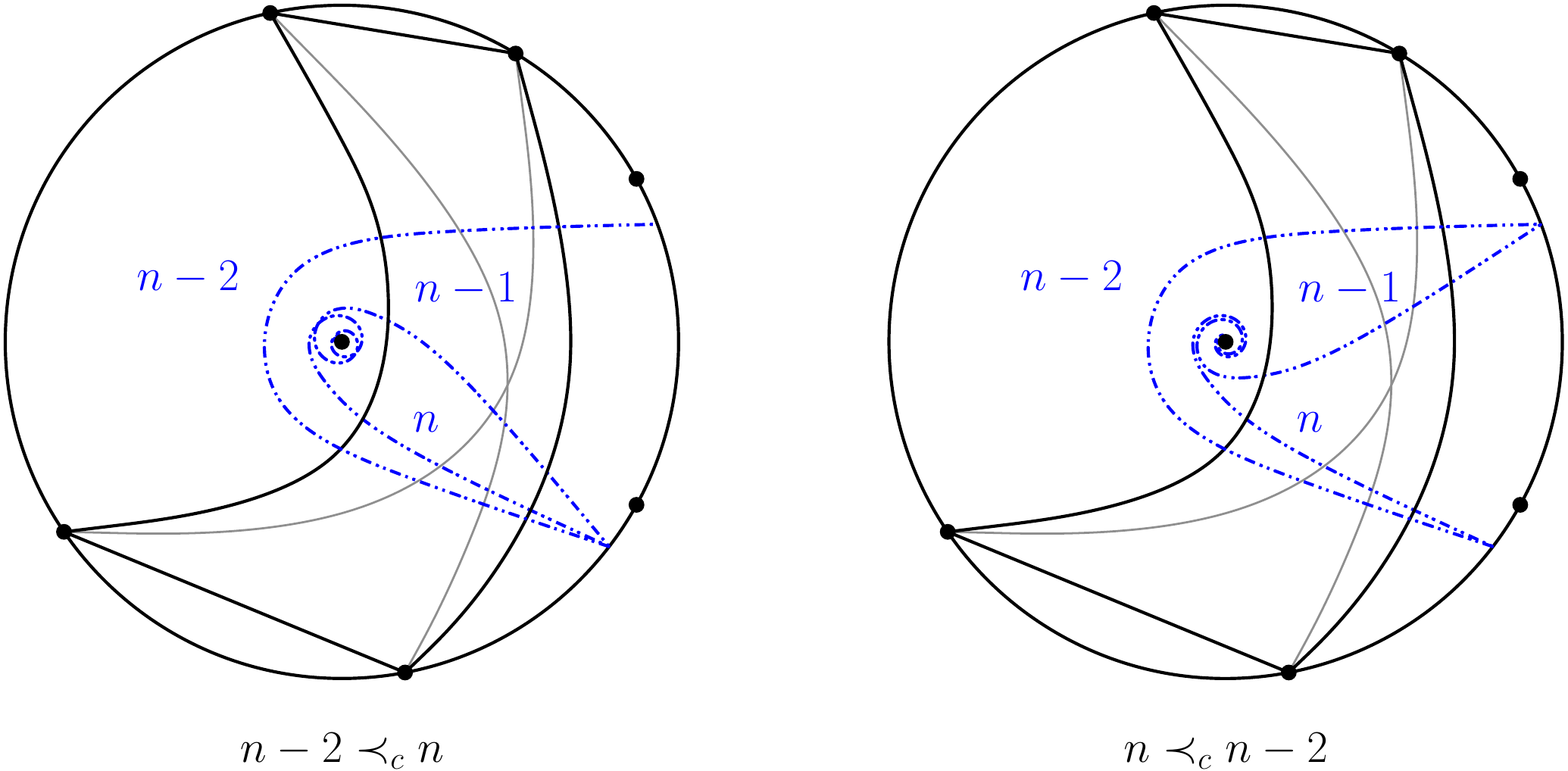}
 \caption{Quadrilaterals in type $D_n$ with both diagonals being chords yield unique solutions to equation~\eqref{eq:system}.}
 \label{fig:D_n-roots}
 \end{figure}

 \begin{Lemma} \label{lem:is-radius}
 In all the cases in which equation~\eqref{eq:system} is satisfied by two distinct positive roots at least one of~$x_\lambda$ and~$x_\mu$ corresponds to a~radius.
 \end{Lemma}
 \begin{proof}
We will show that, if the arcs corresponding to $x_\lambda$ and $x_\mu$ are both chords, then there is a unique positive root satisfying equation~\eqref{eq:system}. The two cases to be considered, namely $n-2 \prec_c n$ and $n \prec_c n-2$, are pictured in Fig.~\ref{fig:D_n-roots}.

Suppose at f\/irst that $n-2 \prec_c n$ and let $\alpha$ be one of the two positive roots satisfying equation~\eqref{eq:system}. By Corollary~\ref{cor:kernel-Dn}, exactly one among the $(n-1)$-st and the $n$-th simple root coordinates of $\alpha$ is $0$, and the other is $1$. In particular, any diagonal of any quadrilateral supporting this exchange relation must give dif\/ferent shear coordinates to the $(n-1)$-st and $n$-th elementary laminations. However, if both diagonals of a quadrilateral are chords, then they do not distinguish the two elementary laminations, so, in particular, each of the two diagonals assigns either~$\pm 1$ or~$0$ to both simultaneously.

Suppose now that $n \prec_c n-2$ and let $\alpha$ be again one of the two positive roots satisfying equation~\eqref{eq:system}. Since one of $\alpha\pm(\alpha_{n-1}+\alpha_n)$ is also a root, the $(n-2)$-nd coordinate of $\alpha$ is $1$ while both the $(n-1)$-st and $n$-th coordinates are simultaneously $1$ or $0$. Thus, any diagonal of any quadrilateral supporting this exchange relation must give shear coordinate $\pm 1$ to the $(n-2)$-nd elementary lamination, and equal values to the $(n-1)$-st and $n$-th ones. Now take any quadrilateral whose diagonals are both chords. If its diagonals give shear coordinate $\pm 1$ to both the $(n-1)$-st and $n$-th elementary lamination, then they assign $\pm 2$ to the $(n-2)$-nd elementary lamination. If instead they give shear coordinate~$0$ to both $(n-1)$-st and $n$-th elementary laminations, then they also assign $0$ to the $(n-2)$-nd elementary lamination.
 \end{proof}

 \begin{Lemma} \label{lem:g-vector-of-radii}
The $\bg$-vector of any cluster variable associated to a radius has exactly one among its $(n-1)$-st and $n$-th fundamental weight coordinates equal to~$0$; the other one is~$\pm 1$.
 \end{Lemma}
\begin{proof}
\cite{FST08,FT12} do not contain an explicit recipe to compute the $\bg$-vector of the cluster variable associated to an arc. An easy rule, though, can be obtained using \cite[equation~(1.13)]{NZ12}: it suf\/f\/ices to ref\/lect our surface and compute the shear coordinates of the elementary lamination corresponding to the desired arc with respect to the initial triangulation (see, e.g.,~\cite[Proposition~5.2]{Re14} or~\cite[Lemma~8.6]{FeTu15}). Since we only care for the last two entries it will suf\/f\/ice to look inside the unique digon in the initial triangulation. We are in the situation depicted in Fig.~\ref{fig:D_n-weights}.

Suppose at f\/irst that $n \prec_c n-2$. It follows immediately from \cite[Fig.~36]{FT12} that, in order for a radial elementary lamination to have non-zero $n$-th shear coordinate, it has either to start from the side of the digon lying on the boundary of the disk and then spiral clockwise to the puncture, or cross the other side and spiral counterclockwise. In either case such a lamination will have $(n-1)$-st shear coordinate equal to $0$. The situation reverses for radial elementary laminations having non-zero $(n-1)$-st shear coordinate.

The case $n-2 \prec_c n$ is identical: the two are related by a f\/lip and the only ef\/fect this has on the last two shear coordinates is to change some of the signs.
 \end{proof}

 \begin{figure}[t]\centering
 \includegraphics[scale=0.6]{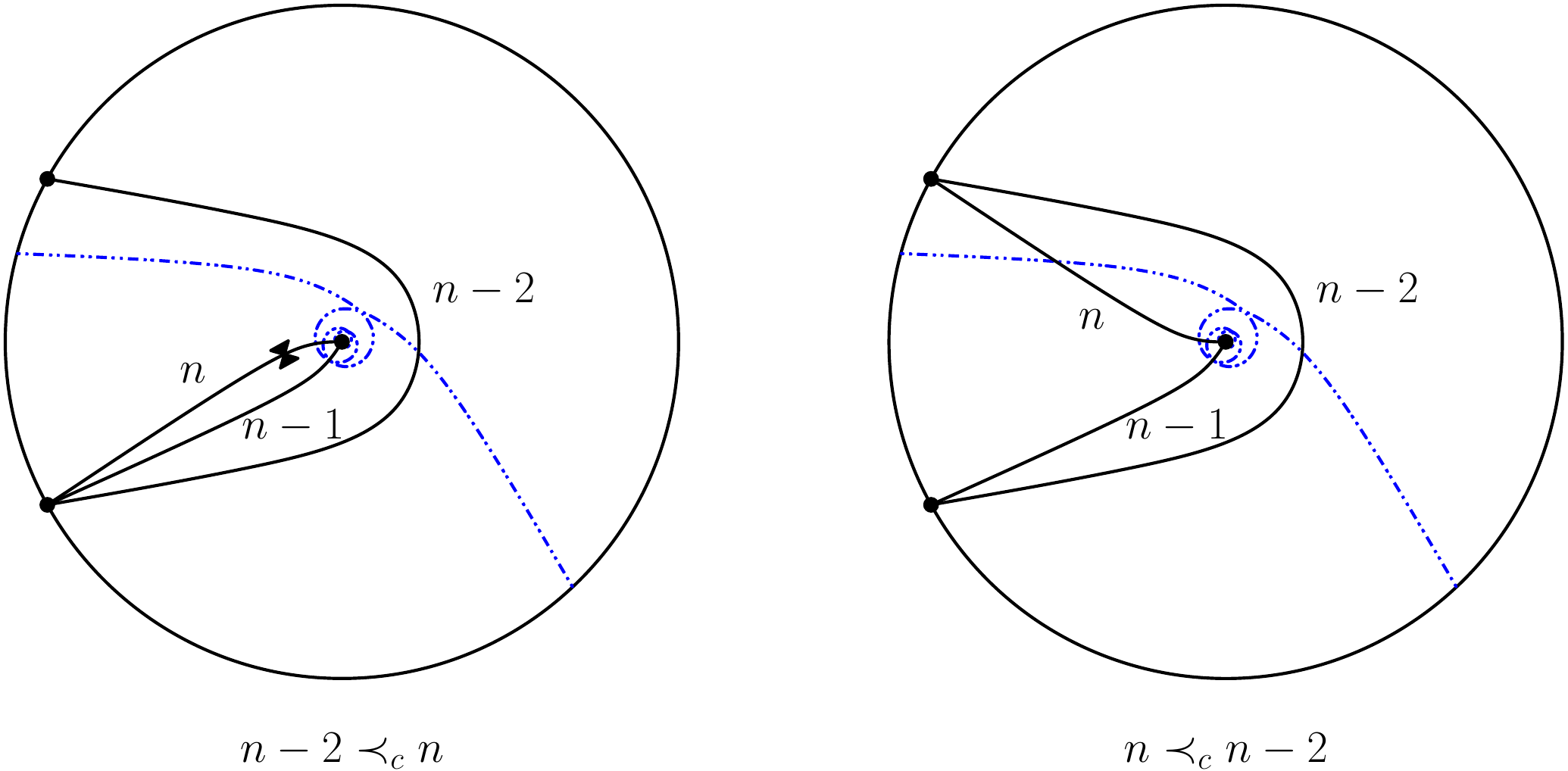}
 \caption{Radial laminations with non-zero $n$-th shear coordinate.}
 \label{fig:D_n-weights}
 \end{figure}

We are f\/inally ready to conclude the proof of Theorem~\ref{thm:main} in type~$D_n$. Suppose~$\lambda$ and~$\mu$ exchangeable weights are such that equation~\eqref{eq:system} has two solutions~$\alpha$ and~$\alpha'$. In particular $\alpha' = \alpha \pm (\alpha_{n-1}-\alpha_n)$ if $n-2 \prec_c n$ and $\alpha' = \alpha \pm (\alpha_{n-1}+\alpha_n)$ if $n \prec_c n-2$.

By Lemma~\ref{lem:is-radius} we can assume that $\lambda$ is the $\bg$-vector of a radius; in particular, by Lemma~\ref{lem:g-vector-of-radii}
 \begin{gather*}
 \langle \lambda, \alpha_{n-1}\pm \alpha_n \rangle = \pm 1
 \end{gather*}
 and thus
 \begin{gather*}
 \langle \lambda, \alpha' \rangle = \langle \lambda, \alpha \rangle \pm 1.
 \end{gather*}
Therefore, since the pairing $\langle\cdot\,,\cdot\rangle$ is integer-valued when computed on (co-)roots and weights, $\alpha$~and~$\alpha'$ cannot both satisfy equation~\eqref{eq:dual}.

\begin{Remark}In view of some ongoing work of the f\/irst author with Nathan Reading it appears that a modif\/ied version of \cite[Propositions~5.1 and~5.2]{Ste13} holds in af\/f\/ine types as well. We expect that Theorem~\ref{thm:main}, or a ref\/ined version of it, could hold there too. The analysis required to establish it, though, will probably be more complicated than the f\/inite case one because the corank of $B_c$ can be as big as $4$ and the argument of Remark~\ref{rk:induction} does not apply.
\end{Remark}

\subsection*{Acknowledgements}
This paper was completed during a visit to Durham University; the f\/irst author would like to thank both Grey College and the Department of Mathematical Sciences for the hospitality received. We are also grateful to Anna Felikson for many fruitful discussions, and to Nathan Reading for his help with some early attempts at Theorem~\ref{thm:main} and for his comments on a~pre\-li\-mi\-nary version of this paper. Finally we would like to thank our anonymous referees for pointing out a f\/law in an earlier version of Lemma~\ref{lem:components} and for several useful comments.

\pdfbookmark[1]{References}{ref}
\LastPageEnding

\end{document}